\documentclass[12pt]{amsart}
\usepackage{geometry}                
\usepackage{graphicx}
\usepackage{amsrefs}
\usepackage{amssymb}
\usepackage[matrix,arrow,ps]{xy}
\usepackage{epstopdf}
\newtheorem{theorem}{Theorem}[section]
\newtheorem{lemma}[theorem]{Lemma}
\newtheorem{cor}[theorem]{Corollary}
\newtheorem{prop}[theorem]{Proposition}
\theoremstyle{remark}
\newtheorem*{remark}{Remark}
\theoremstyle{example}
\newtheorem{example}[theorem]{Example}
\theoremstyle{definition}
\newtheorem{definition}[theorem]{Definition}

\title{Local equivalence problem for Levi flat hypersurfaces}
\date{\today}                                           

\author{Giuseppe Della Sala}
\address{Department of Mathematics\\ \newline University of Vienna\\ Vienna\\ 1090\\ Austria}
\email{giuseppe.dellasala@univie.ac.at, beppe.dellasala@gmail.com}

\begin{document}

\begin{abstract}
In this paper we consider germs of smooth Levi flat hypersurfaces, under the following notion of local equivalence: $S_1\sim S_2$ if their one-sided neighborhoods admit a biholomorphism smooth up to the boundary. We introduce a simple invariant for this relation, which allows to prove some characterizations of triviality (i.e.\ equivalence to the hyperplane). Then, we employ the same invariant to construct infinitely many non-trivial classes, including an infinite family of not equivalent hypersurfaces which are almost everywhere analytic.
\end{abstract}

\maketitle

\section*{Introduction} The question of biholomorphic equivalence is an old and natural problem in the field of several complex variables. From a general point of view, it can be stated as follows: given two (real) submanifolds $M,M'\subset \mathbb C^n$, $0\in M,M'$, establish whether a local biholomorphism $\Phi$ (defined in a neighborhood of $0$) exists such that $\Phi(M)=M'$. The first one to consider this ``local'' problem was Poincar\'e \cite{Po}, who showed that generically, two real submanifolds  will not be biholomorphically equivalent to each other, and that the number of invariants is necessarily infinite.
In the last century the topic has been taken up by many authors, and for several classes of manifolds a complete solution has been obtained. Levi non-degenerate real hypersurfaces in $\mathbb{C}^n$ were classified by Cartan (see \cite{Ca1},\cite{Ca2}) for $n=2$, and independently by Tanaka \cite{Ta} and by
Chern and Moser (who gave a classification in terms of \emph{normal forms} in \cite{CM}) for $n\geq 2$. 
Later on, the question has been considered under hypotheses of increasingly higher degeneracy; for
example, Kol{\'a}{\v{r} \cite{Ko} has provided a normal form for hypersurfaces of finite type in $\mathbb C^2$---an exhaustive review of the subject would be far beyond the scope of this introduction.

It is worth noting that the equivalence problem has been treated--for very good reasons--almost exclusively for {\em real-analytic} $M$ and $M'$. 
Under this hypothesis, the case of \emph{maximal} degeneracy, i.e.\ that of Levi flat hypersurfaces (that is, real hypersurfaces foliated by 
complex hypersurfaces), becomes trivial: real analytic Levi flat
hypersurfaces are in fact locally biholomorphic to the hyperplane $\{ (z,w) \in \mathbb{C}^n\times \mathbb{C} \colon \rm{Im} w =0\} $  
(see e.g. \cite{AF}). We thus turn to the case of merely smooth Levi flat
hypersurfaces; we consider the following notion of equivalence, corresponding to the \lq\lq localized\rq\rq\ version of a biholomorphism, 
smooth up to the boundary,
between two domains of $\mathbb C^n$:

\begin{definition} \label{loceq} Let $\mathcal U\subset \mathbb C^n$ be a domain, and let $S_1, S_2$ be Levi flat hypersurfaces of class $C^k$ ($k\geq 1$)
embedded in $\mathcal U$, $0\in S_i$. We suppose that each $S_i$ divides $\mathcal U$ in two connected components, and we fix $U_i\subset \mathcal
U\setminus S_i$. We say that $(S_1,U_1)$ and $(S_2,U_2)$ are \emph{ $C^k$ locally equivalent} if there exist neighborhoods $\mathcal V_i$ of $0$ in $\mathbb
C^n$ such that, putting $V_i= \mathcal V_i\cap U_i$, there is a biholomorphism $\Phi: V_1\to V_2$ which admits an extension to a $C^k$-CR map from
$S_1\cap \mathcal V_1$ to  $S_2\cap \mathcal V_2$ (which we again
denote by $\Phi$).\end{definition} 

We note that the one-sided neighborhood $U_i$ is only important in order to define the {\em side} of $S_1$ (or $S_2$, respectively); if it has
been fixed, we can just talk about the equivalence of $(S_1,0)$ and $(S_2,0)$. Moreover, in this paper, 
we will mostly deal with the case $k=\infty$ and require that the local equivalence $\Phi$ is also of class $C^\infty$; 
in this case we will simply talk about local equivalence.

Our general aim is to study the behavior of germs of Levi flat hypersurfaces through $0\in \mathbb C^n$ with respect to this equivalence relation. In
particular, we introduce a basic invariant, the property of admitting a foliation by complex hypersurfaces on $\overline U$, holomorphic
on $U$ and ($C^k$) smooth up to $S$. This condition is fulfilled for all real-analytic Levi flat hypersurfaces, and we show that in the smooth case it characterizes the {\em trivial} germs, i.e.\ those which are equivalent to the hyperplane. In Section \ref{triv} we discuss some of the properties of this invariant, and
 we give some  conditions and criteria for the germ of $S$ to be trivial. The general idea that we follow is that it
should be possible to reduce the equivalence problem for Levi flat germs to a (possibly infinite) collection of extension problems in one complex variable.
Then, in Section \ref{autom} we use the same invariant to prove another characterization of triviality, involving the structure of the local automorphism
group of $(S,U)$ (to be defined in the same section). Finally, in Section \ref{nontriv} we show that germs of  
holomorphic foliations can also be employed to find many nonequivalent non-trivial classes: in
particular, we prove that there exist infinitely many non-equivalent hypersurfaces which are real analytic outside the complex hyperplane 
in $S$ passing through $0$.

For simplicity, from now on we fix $n=2$ and we consider coordinates $(z=x+iy, w=u+iv)$; most of the results  are also valid (with
slight modifications) for all $n\geq 2$.

\

\noindent {\bf Acknowledgments}: This paper was written while the author was a post-doc at the University of Vienna. I am very grateful to B.Lamel for posing the question in the first place, for several conversations regarding the problem and for otherwise helping in many crucial ways. I also wish to thank A.Saracco for discussing the topic, and for an useful suggestion about section \ref{nontriv}.

\section{Triviality criteria}\label{triv}
\noindent We say that a germ $(S,U)$ is ($C^k$) \emph{trivial} if it is ($C^k$) equivalent to $\{v=0\}$. In this section, we want to provide a few  conditions that are equivalent or in various ways related to the triviality of $S$, and to show how extension problems in one complex variable get in the picture. Our main tool is the invariant obtained in the following lemma, which is a consequence of Riemann's mapping theorem:
\begin{lemma} \label{inv}
Let $(S,U)$ be a smooth Levi flat germ centered at $0\in \mathbb C^2$. Then the following conditions are equivalent:
\begin{itemize}
\item [(a)] $S$ is trivial;
\item [(b)] there exists a neighborhood $\mathcal V$ of $0$ and a holomorphic foliation $\mathcal F$ in $V=\mathcal V\cap U$ which is smooth up to the boundary.
\end{itemize}
\end{lemma}
\begin{remark}
Clearly, if $\mathcal F$ is as above then $\mathcal F |_S$ coincides with the Levi foliation of $S$. Moreover, $\mathcal F$ is uniquely defined by the previous condition. In fact, suppose that the leaf of $\mathcal F$ passing through the point $p_0=(z_0,w_0)$ is expressed as $\{w=f(p_0,z)\}$, and let $\mathcal G$ be another foliation, satisfying the same condition, such that the leaf through $p_0$ is $\{w=g(p_0,z)\}$. Then, for every $k\in \mathbb N$, $\frac{\partial^k f}{\partial z^k}(p_0)$ and $\frac{\partial^k g}{\partial z^k}(p_0)$ are holomorphic in $p_0$ and smooth up to the boundary; moreover, they coincide along $S$, hence everywere. It follows that $\mathcal F \equiv \mathcal G$.
\end{remark}
\begin{proof}
The implication (a) $\Rightarrow$ (b) is clear, as we can take the image of the foliation induced by $\{w=c\}$, ${\rm Im}\ c\geq 0$. To show that (b) $\Rightarrow$ (a), suppose that $H_0(S)=\{w=0\}$; then the intersection $S\cap \{z=0\}$ is transversal. Moreover, by the Riemann mapping theorem follows that (up to a possible shrink of U) there exists a biholomorphism $$\phi:W\cap \{v>0\}\to U\cap \{z=0\}$$ (where $W$ is a neighborhood of $0$ in $\mathbb C$). It is a well-known (but highly non-trivial) fact that $\phi$ is smooth up to the boundary: see, for example, \cite{Be}.

Observe that (again, after a possible shrinking of $U$) all the leaves of $\mathcal F$ are of the form $\{w=f(\eta,z)\}$, where $\eta\in U\cap \{z=0\}$, $z\in D(0,\varepsilon)$ and $f$ is a holomorphic function, smooth up to $S$, such that $f(\eta,0) = \eta$. Defining
$$\Phi:D(0,\varepsilon)\times (W\cap \{v>0\})\to U, \ \Phi(z,w) = (z, f(\phi(w),z))$$
we have that $\Phi$ is a biholomorphism, smooth up to the boundary (since it is a composition of smooth maps).
\end{proof}
As a first application of the previous characterization, we obtain that (if $S$ is smooth) it is sufficient to check triviality by a continuous map to conclude that $S$ is $C^\infty$ trivial:

\begin{lemma}\label{cont}
Let $(S,U)$ be a smooth Levi flat germ. Then $S$ is (smoothly) trivial if and only if it is $C^0$ trivial.
\end{lemma}
\begin{remark} Even when the germ $S$ is smooth and trivial, a trivializing equivalence $\Phi$ needs not be of class $C^\infty$ (though the normal component of $\Phi$ will necessarily be smooth, see \cite{BL}): for instance one can consider an automorphism $\Phi=(\phi_1(z,w),\phi_2(w))$ of the hyperplane $\{v=0\}$ with $\phi_1$ just continuous up to the boundary (cfr. example \ref{authyp}). Hence, in Lemma \ref{cont} a smooth equivalence to the hyperplane is in general realized by a different $\Phi$.
\end{remark}
\begin{proof}
Assume that $(S,U)$ is $C^0$ trivial; as in Lemma \ref{inv}, we can locally find a foliation $\mathcal F$ which is holomorphic in $U$ and extends continuously up to $S$. With a suitable choice of coordinates, we can write the tangent distribution of $\mathcal F$ as $T(p)=\{w= g(p)(z-z(p)) + w(p)\}$, where $T(p)$ is the (complex) tangent plane of $\mathcal F$ at $p\in \overline U$, and $g(p)$ is holomorphic on $U$ and continuous on $\overline U$. Since $\mathcal F|_S$ coincides with the foliation of $S$ (which is $C^\infty$), $g|_S$ is a smooth $CR$ function. By Plemelj formulae for hypersurfaces (see \cite{HL}, Appendix B), there are two functions $G^+\in \mathcal O(U)\cap C^\infty(\overline U)$, $G^-\in \mathcal O(\mathcal U\setminus \overline U)\cap C^\infty(\mathcal U \setminus U)$ such that $G^+|_S - G^-|_S = g|_S$: it follows that $G(p)=G^+(p)-g(p)$, $p\in U$, is the analytic continuation of $G^-$ to $U$ and hence it is holomorphic in a neighborhood of $0$. In particular $g\in C^\infty(\overline U)$, and the distribution $T(p)$ is smooth. Integrating $T(p)$ we obtain that $\mathcal F$ is smooth up to $\overline U$, and by Lemma \ref{inv} follows that $S$ is trivial.\end{proof}
Denote with $H_p(S)$ the complex tangent space at the point $p\in S$. From now on, unless otherwise stated, we are going to assume that $H_0(S)=\{w=0\}$ (we may even ask that $\{w=0\}\subset S$). Then the leaves of $S$ can be expressed as $$w=g(t,z)=\sum_{k=0}^\infty b_k(t)z^k$$
where $(t,z)$ lies in a neighborhood of $0$ in $\mathbb R\times \mathbb C$, and each $b_k$ is a $C^\infty$, complex-valued function of $t$. We observe, however, that although the sequence $\{b_k\}$ contains all the \lq\lq information\rq\rq\ about $S$, it is of course not a natural choice since it depends on the parametrization chosen. To get around this, let $\gamma = S\cap \{z=0\}$; then $\gamma\subset \mathbb C_w$ is a smooth real curve, parametrized by $t \to b_0(t)$. We define $a_k:\gamma\to \mathbb C$ in the following way:
$$ a_k(p) = b_k(b_0^{-1}(p)),\ p\in \gamma.$$
Let $U_0 = U\cap \{z=0\}$; then $U_0$ is a one-sided neighborhood of $\gamma$ in $\mathbb C_w$. The triviality of $S$ is linked to the extendability properties of the $a_k$'s to $U_0$:
\begin{lemma} \label{inver} Let $S$ be as in Lemma \ref{inv}. If $S$ is trivial, then each $a_k$, $k\geq 1$, is the boundary value of a holomorphic function defined on $\mathcal W\cap U_0$ and smooth up to the boundary, where $\mathcal W$ is a fixed neighborhood of $0$ in $\mathbb C^2$.
\end{lemma}
\begin{proof}
Let $\Phi:U\cap \mathcal W\to D(0,\varepsilon)\times (\mathcal W'\cap \{v>0\})$ be a biholomorphism smooth up to the boundary, and let $\mathcal F$ be the foliation induced by the inverse images of $\{w=c\}$. Then we may suppose that all the leaves of $\mathcal F$ are graphs over $\mathbb C_z$, i.e.\ the leaf passing through $(z_0,w_0)$ has the form $\{w=g(w_0,z)\}$ with $g(w_0,z_0)=w_0$. Then clearly $\frac{\partial^k g}{\partial z^k}(w_0,z_0)$ is in $\mathcal O(\mathcal W\cap U)\cap C^\infty (\overline {\mathcal W\cap U})$, so that
$$a_k(w) = k! \frac{\partial^k g}{\partial z^k}(0,w),\ w\in U_0$$
is the required extension of $a_k$.
\end{proof}
\begin{remark}
Lemma \ref{inver} gives a necessary condition for the triviality of $S$; we observe, however, that the converse is not true, even when each $a_k$ extends holomorphically to a fixed neighborhood on both sides. For example, let $S$ be defined as the image of $\mathbb C\times \mathbb R \ni (z,u)\to (z,g(z,u))\in \mathbb C^2$, with $g(z,u) = \sum_{n=0}^\infty a_n(u) z^n$ and
$$a_0(u) = u, \ \ a_n(u) = \left ( \frac{i}{u-i} \right )^{n^2} \  {\rm for}\  n\geq 1.$$
Observe that $|a_n(u)|\leq 1$ for all $u\in \mathbb R$, $n\geq 1$, so that $g(z,u)$ is convergent (uniformly on compact subsets of $\mathbb C\times \mathbb R$). Moreover, each of the $a_n(u)$ is an analytic function; nevertheless, $g(z,u)$ is not, since it does not admit a holomorphic extension to a neighborhood of $0$ in $\mathbb C^2$. In fact, denote by $a_n(w)$ the holomorphic extension of $a_n(u)$ to $\mathbb C\setminus \{i\}$; then we have
$$a_n(iv) = (v-1)^{-n^2}, \ n\geq 1,\  0<v<1  \Rightarrow \limsup_{n\to +\infty}\sqrt[n]{|a_n(iv)|} = +\infty.$$
Hence the radius of convergence of the series $\sum_n a_n(iv) z^n$, for fixed $v\in(0,1)$, is $0$.
We also remark that it is possible to modify the previous example in order to obtain real valued coefficients: let
$$b_n(w) = \frac{1}{2}(a_n(w) + a_n(-w)) - \frac{i}{2}(a_n(w) - a_n(-w))$$
for $w\in \mathbb C\setminus \{i,-i\}$. Then it can be seen that $b_n(u)$ is real-valued and $|b_n(u)|\leq \sqrt{2}$ for $u\in \mathbb R$; moreover, $|b_n(iv)|\geq \frac{\sqrt{2}}{2}|a_n(iv)|$ for $0<v<1$, so that the series $\sum_n b_n(w)z^n$ has the desired properties.

\end{remark}

In some particular cases, the previous simple lemma has already strict consequences on the structure of a trivial Levi flat germ $S$. The following corollary is an example:
\begin{cor}
Let $S$ be a Levi flat hypersurface, and suppose that it contains the germ $R$ (with $0\in R$) of a real analytic, totally real submanifold of dimension $2$. Then $S$ is trivial if and only if it is real analytic.
\end{cor}
\begin{proof}
As it is well known (see for example \cite{BER}), up to a holomorphic change of coordinates in a neighborhood of $0$ we can assume $R = \{v=0, y=0\}$, so that $S\cap \{z=0\} = \{v=0\}=\gamma \subset \mathbb C_w$ and $S\cap \{y=0\} = R$. As before, we write (the leaves of) $S$ as $w=g(t,z) = \sum_n a_n(t)z^n$, $t\in \gamma$: observe that, since $S\cap \{y=0\} = R \subset \{v=0\}$, the series $\sum_n a_n(t) x^n$ is real valued (for $x$ in a neighborhood of $0$ in $\mathbb R$), which implies that each $a_k$ is real valued. If $S$ is trivial, by the proof of Lemma \ref{inver} $a_k$ is the smooth boundary value  of a holomorphic function $a_k(w)$ (defined on a fixed neighborhood of $0$ in $\{v>0\}$) \emph{and} the series  $g(z,w)= \sum_n a_n(w)z^n$ converges. By the Schwarz reflection principle follows that each $a_k$ is real analytic and the series $g(z,w)$ extends to a neighborhood of $0$ in $\mathbb C^2$; thus $S$ (which is the graph of $g(z,u)$) is real analytic.

Vice versa, if $S$ is real analytic then it is locally equivalent to a hyperplane in a sense much stronger than the one we are employing (see \cite{AF}).
\end{proof}
On the other hand, if the $a_k$ are all suitably extendable it follows the triviality of $S$. For instance we have
\begin{lemma}
Suppose that there exist a neighborhood $\mathcal W$ of $0$ in $\mathbb C$ and a constant $C>0$ such that each $a_k$ extends to $W = \mathcal W\cap U_0$ as a holomorphic, smooth up to the boundary function, and $\|a_k\|_W\leq C^k \|a_k\|_{\gamma}$ (where $\| \cdot \|_K$ is the $\sup$ norm on $K$). Then $S$ is trivial.
\end{lemma}
\begin{proof}
Since the functions $g(\cdot,z)$ are defined for $|z|<\varepsilon$ and are uniformly bounded, by the Cauchy estimates there exist constants $M,R>0$ such that $|a_k(p)|\leq M/R^k$ for all $p\in \gamma$. By hypothesis, each $a_k$ extends in such a way that $|a_k(w')|\leq M (C/R)^k$ for all $w'\in W$, so that all the $g(w',z)$ have a positive, uniform radius of convergence. Hence, by the same arguments as in Lemma \ref{cont}, the function $g(w',z)$ is holomorphic in $W\times D$ (where $D\subset \mathbb C_z$ is a neighborhood of $0$) and smooth up to $\gamma\times D$. Let $\Phi:D\times  W \to \mathbb C^2$ be defined as $$\Phi(z,w')= (z, g(w',z)).$$ Observe that, because of our choice of $g$, the differential of $\Phi|_{D\times \gamma}$ has maximal rank; then by the holomorphicity of $\Phi$ follows that  the Jacobian of (a smooth extension to $D\times \mathcal W$ of) $\Phi$ does not vanish. Thus we locally obtain a foliation of $U$, whose leaves $\Sigma_c$ ($c\in W$) can be written as $\Sigma_c = \Phi(\cdot,c)$, which is holomorphic and smooth up to the boundary; by Lemma \ref{inv} follows the triviality of $S$.
\end{proof}

\section{Local automorphisms}\label{autom}
Let $(S,U)$ be as before. We say that $S$ admits a \emph{local automorphism} if the germs generated in $p_1, p_2\in S$ are locally equivalent according to Definition \ref{loceq}. The local automorphisms form a pseudogroup, in the sense that if $\Phi_1:V_1\to U$ and $\Phi_2:V_2\to U$ are local equivalences then $\Phi_2\circ\Phi_1$ is only defined on $V_1\cap \Phi_1^{-1}(V_2)$. Observe that the restriction $\Phi|_S$ of a local automorphism $\Phi$ to $S$ is a local $CR$ automorphism of $S$, i.e. it preserves its foliation and it is holomorphic along the fibers. The restriction operator $\Phi|_S$ is clearly injective (as the components of $\Phi$ are holomorphic), but it is not surjective (cfr. Example \ref{authyp}).

 If $0\in S$, we are mainly concerned about the germ of $(S,U)$ in $0$, hence we will restrict our attention to those local automorphisms $\Phi$ which are defined on a domain $V\ni 0$, and are invertible in the same category (i.e. $0\in \Phi(V)$). We denote by $Aut(S,U)$ the set of such automorphisms.
 \begin{example}\label{authyp}
\emph {Let $S$ be the real hyperplane $S=\{v=0\}\subset \mathbb C^2$, $U=\{v>0\}$. Then the elements of $Aut(S,U)$ must preserve the holomorphic foliation of $U$ by $\{w=c\}$, and are of the form}
 $$\Phi(z,w)=(\phi_1(z,w),\phi_2(w))$$
 \emph{where $\phi_2$ is a local automorphism of the upper half-plane $H\subset \mathbb C_w$ around $0$, and $\phi_1$ is holomorphic in $(z, w)$, smooth up to the boundary, and with $\frac{\partial}{\partial z}\phi_1(0,0)\neq 0$. Observe that, by the Schwarz reflection principle, $\phi_2$ is actually a holomorphic function in a neighborhood of $0$ in $\mathbb C_w$, real valued along $\{v=0\}$. It is also clear that the $CR$ automorphisms of $S$ are not, in general, extendable to $U$; if $\phi_1(z,u) = \sum_{k} a_k(u)z^k$, then in order for this to happen  each $a_k$ must be holomorphically extendable to (a neighborhood of $0$ in) the upper half plane $H$.}
\end{example}
However, we will mostly be interested in a  subset of $Aut(S,U)$. Let $\mathcal F$ be the Levi foliation of $S$: we denote by $Aut'(S,U)$ the subset of those automorphisms $\Phi$ whose restriction to $S$ fix (rather than just preserving) $\mathcal F$, i.e.\ which act in such a way that $\Phi(F)\subset F$ for every leaf $F\in \mathcal F$.
\begin{example}
\emph{Let $(S,U)$ be as in the previous example; then for any $\Phi\in Aut'(S,U)$ we have}
$$\Phi(z,w)=(\phi_1(z,w),w)$$
\emph{(with $\phi_1$ as before), since in this case the restriction of $\phi_2$ to the real axis is the identity.}
\end{example}
We will regard $Aut(S,U)$, $Aut'(S,U)$ as pseudogroups, too, where the composition $\Phi_2\circ\Phi_1$ is defined whenever $0\in \Phi_2(V_2\cap \Phi_1(V_1))$.

As the examples above show, when $S$ is trivial there are actually \lq\lq many\rq\rq\ $CR$ automorphisms that extend to $Aut'(S,U)$. We can construct other examples in the following way: suppose that $H_0(S)=\{w=0\}$, and, for $c\in \mathbb C$, $|c|$ very small, let $\Gamma_c = (\gamma^1_c,\gamma^2_c)$ be the (uniquely determined) $CR$ automorphism such that $\Gamma_c$ fixes the Levi foliation $\mathcal F$ and $\gamma^1_c(z,w)=z+c$. Then it is easy to see (for example by using the holomorphic foliation of $U$) that all the $\Gamma_c$'s extend to $Aut'(S,U)$ whenever $S$ is trivial. Observe that the $CR$ automorphisms $\Gamma_c$ commute, and they form a local one (complex) parameter pseudogroup, i.e.\ $\Gamma_{c_2}\circ\Gamma_{c_1}=\Gamma_{c_1+c_2}$; moreover, their coefficients are holomorphic in $c$. Our aim is to prove that all this can only happen when $S$ is trivial.

\begin{prop}
Let $(S,U)$ be a smooth Levi flat germ centered at $0\in \mathbb C^2$. Then the following conditions are equivalent:
\begin{itemize}
\item [(a)] $S$ is trivial;
\item [(c)] there exists a one (complex) parameter pseudogroup $\Gamma_c$ ($c\in D_\varepsilon = \{z\in \mathbb C:|z|<\varepsilon\}$) of $CR$ automorphisms of $S$, holomorphic in $c$ with $\frac{\partial}{\partial c}\Gamma_0(0)\neq (0,\star)$, such that each $\Gamma_c$ extends to $Aut'(S,U)$.
\end{itemize}
\end{prop}
\begin{proof} We already observed that $(a)\Rightarrow (c)$. To prove the converse we first fix some notations: for any $c\in D_\varepsilon$, let $\widetilde \Gamma_c$ be the the (unique) extension of $\Gamma_c$ to an element of $Aut'(S,U)$, and for any $p\in \overline{U}$ let $\Sigma_p$ be the orbit of $p$ under the action of the $\widetilde \Gamma_c$'s:
$$\Sigma_p = \{\widetilde \Gamma_c(p): c\in D_\varepsilon\}.$$
Clearly, since $\frac{\partial}{\partial c}\Gamma_0(0)\neq (0,\star)$, up to a possible shrinking of $U$ we have $\Sigma_p=F_p$ for all $p\in S$, where $F_p$ is the leaf of the Levi foliation $\mathcal F$ passing through $p$. We claim that $\Sigma_p$ is a germ of complex curve also for $p\notin S$: we will show that the map $c\to \widetilde \Gamma_c(p)$ is holomorphic in $c$ (note that, a priori, $\Gamma_c$ may not even extend continuously in $c$). This is a consequence of the following lemma, which is probably well known; since it is not hard, we are going to give a proof in any case.
\begin{lemma}
Let $\lambda\ni 0$ be a smooth real curve embedded in a domain $\mathcal U\subset \mathbb C$, and let $U$ be a connected component of $\mathcal U\setminus \lambda$. Let $F:D_\varepsilon\times \overline U\to \mathbb C$ be a bounded function such that $F(c,\cdot)\in \mathcal O(U)\cap C^\infty(\overline U)$ for any fixed $c\in D_\varepsilon$, the first derivatives of $F(c,\cdot)|_\lambda$ are bounded (uniformly with respect to $c$), and $F(\cdot,z)\in \mathcal O(D_\varepsilon)$ for any fixed $z\in \lambda$. Then $F\in \mathcal O(D_\varepsilon\times U) \cap C^\infty(D_\varepsilon\times \overline U)$.
\end{lemma}
\begin{proof}
Let $\mathcal V$ be a small neighborhood of $0$ in $\mathbb C$, set $V_1=\mathcal V\cap U$, $V_2=\mathcal V \setminus \overline V_1$ and choose a real valued $\chi\in C^\infty_0(\lambda\cap\mathcal V)$ such that $\chi\equiv 1$ in a neighborhood of $0$. For any $c\in D_\varepsilon$, we define
$$F^+_c(z) = \frac{1}{2\pi i} \int_\lambda \frac{\chi(\zeta) F(c,\zeta)}{z-\zeta}d\zeta, \ z\in V_1,$$
$$F^-_c(z) = \frac{1}{2\pi i} \int_\lambda \frac{\chi(\zeta) F(c,\zeta)}{z-\zeta}d\zeta, \ z\in V_2.$$
According to the Plemelj formula, the holomorphic functions $F^+_c$ and $F^-_c$ extend continuously up to $\lambda$ for any $c\in D_\varepsilon$ and
$$F^+_c(z) - F^-_c(z) = \pm \chi(z)F(c,z), \ z\in \lambda,$$
(the sign depending on the choice of the orientation of $\lambda$; we shall assume that it be positive). Clearly, $F^+_c$ and $F^-_c$ are holomorphic in the parameter $c$. Moreover, if we define
$$G_c(z) = \begin{cases} F^+_c(z) - F(c,z), & z\in V_1 \\ F^-_c(z), & z\in V_2\end{cases}$$
then $G_c$ is holomorphic on $V_1$, $V_2$ and continuous up to a neighborhood of $0$ in $\lambda$, hence it is holomorphic on a neighborhood $\mathcal W$ of $0$ in $\mathbb C$.
We claim that $G_c$ is bounded, uniformly in $c$; since $F(c,z)$ is bounded by hypothesis, we must show that this is also the case for $F^+_c$ and $F^-_c$. Let $\widetilde F(c,z)$ be a smooth extension of $\chi F(c,z)|_{D_\varepsilon \times \lambda}$ to $D_\varepsilon \times U$. Since by hypothesis $F(c,z)|_{D_\varepsilon \times \lambda}$ is uniformly bounded, along with its first derivatives, with respect to $c$, we can choose a $\widetilde F$ satisfying the same property. Then we can express $F^+_c$, $F^-_c$ as follows:
$$F^+_c(z) = \frac{1}{2\pi i} \int_\lambda \frac{\chi(\zeta) F(c,\zeta)-\widetilde F(c,z)}{z-\zeta}d\zeta + \widetilde F(c,z), \ z\in V_1 $$
$$F^-_c(z) = \frac{1}{2\pi i} \int_\lambda \frac{\chi(\zeta) F(c,\zeta)-\widetilde F(c,z)}{z-\zeta}d\zeta, \ z\in V_2$$
(again, the fact that the term $\widetilde F(c,z)$ appears in the expression of $F^+_c$ rather than $F^-_c$ depends on the choice of the orientation of $\lambda$). Since the first derivatives - and thus by integration the difference quotients - of $\widetilde F$ are uniformly bounded, we get the same for $F^+_c$ and $F^-_c$.

Now, we claim that $G_c(z)$ is continuous in $c$ for any $z\in \mathcal W$. In fact, suppose the contrary, and for some $z_0\in \mathcal W\setminus V_2$ and $\epsilon>0$ let $\{c_k\}_{k\in \mathbb N}$ be a sequence such that $c_k\to c_0$ and $|G_{c_k}(z_0) - G_{c_0}(z_0)|>\epsilon$ for all $k\in \mathbb N$.
Since $G_c$ is bounded, by Montel's theorem there is a subsequence (we denote it again by $c_k$) such that $G_{c_k}$ converges to a holomorphic function $G'$ uniformly on compact sets; but since $G_c(z)$ is continuous in $c$ for all $z\in \mathcal W\cap V_2$, we must have $G'\equiv G_{c_0}$, a contradiction.

By the continuity of $G_c$, for any $z\in \mathcal W$ and any simple closed curve $\gamma\subset D_\varepsilon$ we can define
$$I_\gamma(z) = \int_\gamma G_c(z) dc.$$
Observe that $I_\gamma$ is holomorphic in $z$ and vanishing on $\mathcal W\cap V_2$, hence $I_\gamma\equiv 0$ on $\mathcal W$. By Morera's theorem, it follows that $G_c(z)$ is holomorphic in $c$ also for $z\in \mathcal W\cap V_1$, and so is $F(c,z)= F^+_c(z) - G_c(z)$. By reiterating the same argument, one obtains that $F(c,z)$ is also holomorphic for $z\in U$.
\end{proof}
Now, let $F_1(c,p)$, $F_2(c,p)$ be the components of $\widetilde \Gamma_c$: observe that our hypotheses imply that these functions are uniformly bounded  (since $\widetilde \Gamma_c(p)\in U$ for all $p$ in its domain of definition), along with the first derivatives of their restriction to $D_\varepsilon \times S$ (up to a possible shrinking of the domain, since $\Gamma_c(z)$ is smooth in $(c,z)$). Applying the previous lemma to the restriction of $F_1$, $F_2$ to $U\cap \{z=const.\}$ we obtain that $\widetilde \Gamma_c(p)$ is holomorphic in  $(c,p)\in D_\varepsilon\times U$ and smooth up to the boundary. The smoothness of $\widetilde \Gamma_c$, in particular, implies that $\frac{\partial}{\partial c}F_1(c,p)\neq 0$ for $p$ close enough to $0\in \mathbb C^2$, hence $\Sigma_p$ is a regular complex curve which is a graph over a small enough disk $D\subset \mathbb C_z$; notice that we can choose a $D$ independent from $p$ in a neighborhood of $0$ in $U$.

Let $U_0=U\cap \{z=0\}$; by the previous remarks, we have that (up to shrinking $U$) $U\subset \bigcup_c \widetilde \Gamma_c(U_0)$. For $p\in U_0$, $p=(0,w)$, we set $\Sigma_w=\Sigma_p$; then the union of the $\Sigma_w$ fills up $U$. Moreover, we claim that $\Sigma_{w_1}\cap\Sigma_{w_2} = \emptyset $ when $w_1\neq w_2$. In fact, observe that the set of the $\widetilde \Gamma_c$ is automatically a pseudogroup since the restrictions $(\widetilde \Gamma_{c_2} \circ \widetilde \Gamma_{c_1})|_S$ and $\widetilde \Gamma_{c_1+c_2}|_S$ coincide for all $c_1,c_2\in D_\varepsilon$. Suppose, then, that for $w_1\neq w_2$ there exists $q\in \Sigma_{w_1}\cap \Sigma_{w_2}$; in other words, there are $c_1,c_2\in D_\varepsilon$ such that $\widetilde \Gamma_{c_1}((0,w_1))= q = \widetilde \Gamma_{c_2}((0,w_2))$. This implies $\widetilde \Gamma_{c_1-c_2}((0,w_1)) = (0,w_2)$, which means that $\Sigma_{w_1}$ is not a graph over the disk $D\subset \mathbb C_z$; this contradicts our choices of $U$ and $U_0$.

By the previous arguments, for every $p\in U$ (where $U$ is possibly shrinked) we can define $\pi(p)=w$ to be the (uniquely determined) $w\in U_0$ such that $p\in \Sigma_w$. The function $\pi$ is then a holomorphic submersion $U\to U_0$ whose level sets coincide with the $\Sigma_w$; it follows that the $\Sigma_w$ form a holomorphic foliation of $U$ which is smooth up to the boundary. By Lemma \ref{inv} we have that $S$ is trivial.

\end{proof}

\section{Some non-trivial classes}\label{nontriv}
We want to show that the invariant defined in Section \ref{triv}, which identifies germs equivalent to the hyperplane, can also be employed to find infinitely many non-trivial classes $(S,U)$ that are not equivalent to each other. The following example shows one of the simplest ways to accomplish that; later, we are going to prove that there are infinitely many non-equivalent hypersurfaces $S$ that are analytic almost everywhere.
\begin{example}\label{aleph0}
\emph{Consider Levi flat hypersurfaces $S$, written as a $1$-parameter family of complex lines in the following way:
$$ S= \bigcup_t \{w = t + g(t)z\}, = \cup_t L_t$$
where  $g$ is a smooth complex valued germ such that $g(0)=0$. Alternatively, $S$ can be written as the image of the following map $\psi:\mathbb C \times \mathbb R\to \mathbb C^2$,
$$\psi(\zeta,t) = (\zeta, t + g(t)\zeta)$$
(where $(\zeta,t)$ belongs to a neighborhood of $(0,0)$ in $\mathbb C\times \mathbb R$), and since the differential of $\psi$ has rank $3$ in $(0,0)$ for any choice of $g$ we have that $S$ is in fact (a germ of) a regular Levi flat hypersurface of $\mathbb C^2$.}

\emph{Note that, if $S$ and $S'$ are the hypersurfaces corresponding to $g$ and $g'$, any $CR$ diffeomorphism $S\to S'$ induces a local diffeomorphism $\rho$ of $\mathbb R$ such that the leaf $L_t$ is mapped to $L'_{\rho(t)}$. Hence, by the arguments of section \ref{triv}, any local equivalence $\Phi:(S,U)\to (S',U')$ induces a diffeomorphism $\rho$ such that the local extendability properties of $g(t)$ and $g'(\rho(t))$ are the same for any $t$ in a neighborhood of $0$ in $\mathbb R$.}

\emph{Now, let $g_1$ be locally extendable in the neighborhood of all the points except $0$ (for example, analytic in $\mathbb R\setminus \{0\}$), and let $g_2$ be a germ without this property, but such that for every neighborhood $U$ of $0$ there is $t\in U$ such that $g_2$ is extendable around $t$. Finally, choose a $g_3$ that does not satisfy the latter (for example, a smooth, nowhere analytic real germ). Then, because of the observations above, the respective $S_1$, $S_2$ and $S_3$ are not equivalent to each other.}

\emph{The same kind of example can be modified to show the existence of a cardinality of continuum of classes. In fact, choose a \lq\lq Morse code\rq\rq\ set $M\subset \mathbb R$, made up by a sequence of points and segments which accumulate to $0$, and choose a function $g$ which is analytic outside $M$ and nowhere analytic in $M$. Then, in order for the relative hypersurfaces $S$ and $S'$ to be equivalent, the sequences $M$ and $M'$ corresponding to $g$ and $g'$ must be diffeomorphic (in a monotone way) in a neighborhood of $0$; clearly, there are infinitely many (in fact $2^{\aleph_0}$) such equivalence classes of sequences. (I thank A.Saracco for this observation). }
\end{example}
Example \ref{aleph0} shows that there are infinitely many classes of the type that we have called $S_2$; however, it also suggests that they may prove difficult to deal with. Now, we focus on Levi flat hypersurfaces of the first type, in which the method of example \ref{aleph0} doesn't allow to further distinguish subclasses.
\begin{definition}
Let $(S,U)$ be a Levi flat germ, and let $C$ be the leaf passing through $0$. We say that $S$ is \emph{quasi trivial} if it is trivial everywhere except along $C$.
\end{definition}
By the observations of the first section, the previously defined hypersurfaces are endowed with a holomorphic extension of the Levi foliation, except along $C$. Our aim is to use these partial extensions to define invariants (under local equivalence), and to show that the set of quasi trivial hypersurfaces contains infinitely many classes.

For this purpose, we define the following family of hypersurfaces:
$$S_n = \bigcup_t \{w= t +g_n(t)z \}$$
where
$$g_n(t) = e^{-\frac{1}{t^{2n}}}.$$
We claim that
\begin{prop}\label{many}
The hypersurfaces $S_n$ are not equivalent to each other.
\end{prop}
In order to prove Proposition \ref{many}, our main idea is to exploit the fact that the partially defined foliations in a neighborhood of $S_n$ \lq\lq extend\rq\rq, albeit in a singular way, to the whole $U_n$. These extensions are invariants for local equivalence, but have different topological behaviours for different $n$'s.
\begin{lemma}\label{smoke}
Let $(S_1,U_1)$ and $(S_2,U_2)$ two Levi flat germs of hypersurfaces, and let $\mathcal H = \mathbb C \times H = \{(\zeta,\eta)\in \mathbb C^2: {\rm Im}\ \eta>0 \}$. Let $\mathcal U$ be a neighborhood of $0$ in $\mathbb C^2$, and let  $\psi_j:\mathcal H\cap \mathcal U\to U_j$, $j=1,2$, be holomorphic maps satisfying
\begin{itemize}
\item $\psi_j$ is a local biholomorphism;
\item there exist $p_j\in \{{\rm Im}\ \eta=0\}$ and neighborhoods $\mathcal V_j$ of $p_j$ in $\mathbb C^2$ such that $\psi_j|_{\mathcal V_j\cap \mathcal H}$ extends smoothly up to the boundary.
\end{itemize}
Suppose that $S_1$ and $S_2$ are locally equivalent by $\Phi:U_1\to U_2$, and suppose that $\Phi(\psi_1(\mathcal V_1))\cap \psi_2(\mathcal V_2)\neq \emptyset$. Then for any $c_1\in \mathbb C$ with $|c_1|<\varepsilon$ and ${\rm Im}\ c_1>0$ there exists $c_2\in \mathbb C$ such that
$\Phi(\psi_1(\{\eta = c_1\}))\subset \psi_2(\{\eta= c_2\})$.
\end{lemma}
\begin{proof}
Let $V$ be the set of the $c_1$  for which there exists $\varepsilon > 0$ such that the statement of the lemma is valid in $D(c_1,\varepsilon)$. $V$ is open by definition: we want to show that it is a non-empty, closed subset of $\mathcal U\cap  H$.

First, choose $(\zeta_1,c_1)\in \mathcal V_1\cap \mathcal H$. By uniqueness of the extension (see the remark after Lemma \ref{inv}) we have that $\Phi(\psi_1(\{\eta=c_1\}\cap \mathcal V_1))\subset \psi_2(\{\eta= c_2\})$ for some $c_2$; by analytic continuation, this holds for the whole image of $\psi_1(\{\eta=c_1\})$, hence $V$ is non-empty.

To show that $V$ is closed, let $c_1$ be a cluster point for $V$ and pick $\zeta_1$ such that $p_1= (\zeta_1,c_1)\in \mathcal U\cap \mathcal H$. Let $q_1=\psi_1(p_1)$ and $q_2=\Phi(q_1)$. Let $\{p_1^n= (\zeta_1^n, c_1^n)\}_{n\in \mathbb N}$ be a sequence such that $c_1^n\in V$ and $p_1^n \to p_1$, and define the corresponding sequences $\{q_1^n\}$, $\{q_2^n\}$. Then $c_2\in \mathbb C$ is uniquely defined in such a way that $\psi_2^{-1}(\Phi(\psi_1(\{\eta = c_1^n\})))\to \{\eta=c_2\}$. Choose $p_2=(\zeta_2,c_2)$ such that $\psi_2(p_2)=q_2$; by hypothesis, there is a local inverse $\psi_2^{-1}$ in a neighborhood of $q_2$ such that $\psi_2^{-1}(q_2)=p_2$. Then $\Psi:=\psi_2^{-1}\circ\Phi\circ\psi_1$ is a well defined biholomorphism between a neighborhood $W$ of $p_1$ and a neighborhood of $p_2$. Let $\Psi(\zeta,\eta) = (\alpha(\zeta,\eta), \beta(\zeta,\eta))$; then, by construction, there exists an open subset $W'\subset W$ such that
$$ \frac{\partial \beta}{\partial \zeta}(\zeta, \eta)\equiv 0\ {\rm for}\  (\zeta,\eta)\in W'.$$
It follows that $\partial \beta/\partial \zeta \equiv 0$ on $W$, i.e.  $\Phi(\psi_1(\{\eta=c_1' \}\cap W))\subset \psi_2(\{\eta= c_2'\})$ for $c_1'$ in a neighborhood of $c_1$. As before, by analytic continuation we have the same for the whole $\psi_1(\{\eta=c_1'\})$, hence $c_1\in V$.
\end{proof}
\begin{remark}
If we assume that $\psi_j(\{\eta= c_j\})\neq \psi_j(\{\eta= c_j'\})$ for $c_j\neq c_j'$, $j=1,2$, then the previous lemma gives in fact an invertible map $c_1\to c_2(c_1)$ defined in a neighborhood of $0$ in $H=\{{\rm Im}\ \eta >0\}$. It is not difficult to see that in this case the map is in fact a biholomorphism.
\end{remark}
\begin{proof}(Proposition \ref{many})  With the notations of Lemma \ref{smoke}, let $\psi_n:\mathcal H\to U_n$ be defined as
$$\psi_n(\zeta,\eta) = (\zeta,  \eta +g_n(\eta)\zeta)$$
where
$$g_n(\eta) = e^{-\frac{1}{\eta^{2n}}}$$
is well defined for ${\rm Im}\ \eta >0$. The images of $\{\eta = c\}$ by the map $\psi_n$ extend the foliation of $S_n$ to a foliation of a neighborhood of $S_n\setminus \{w=0\}$ in $U_n$; in the whole of $U_n$, they can be simply regarded as a holomorphic $1$-parameter family of complex lines. Let $\Phi$ be a local equivalence between $S_{n_1}$ and $S_{n_2}$. By Lemma \ref{smoke}, $\Phi$ sends the respective $1$-parameter families one into another, i.e. there exists a biholomorphism (in a neighborhood of $0$) $\phi$ such that the diagram
  $$\xymatrix{
  H \ar@{->}[rr]^{ \widetilde \psi_{n_1}} \ar@{<->}[dd]_{ \phi}
      && \widetilde U_{n_1} \ar@{<->}[dd]^{ \Phi}    \\ \\
 H \ar@{->}[rr]_{\widetilde \psi_{n_2}} && \widetilde U_{n_2}
 }$$
commutes, where we denote by $\widetilde U_{n}$ the space of the complex lines in $U_n$ and by $\widetilde \psi_{n}$ the map $c\to \widetilde \psi_n(c) = \psi_n(\{\eta = c\})$.
Now, consider the  \lq\lq central\rq\rq\ leaf $C=\{w=0\}$; we are interested in the behaviour of the families $\widetilde \psi_{n_j}(H)$ with respect to $C$. In particular, let $B_\varepsilon$ be an $\varepsilon$-neighborhood of $0$ in $C$, and consider
$$H_{n,\varepsilon}= \{c_n \in H: \widetilde\psi_n(c)\cap C\in B_\varepsilon\}.$$ 
The following lemma shows that the previously defined sets are not homeomorphic for different $n$'s:
\begin{lemma}\label{conn}
$H_{n,\varepsilon}$ has $n$ connected components for small enough $\varepsilon$. 
\end{lemma}
\begin{proof}
Observe that the set $H_{n,\varepsilon}$ is explicitly given by
$$H_{n,\varepsilon}=\{\eta\in \mathbb C: {\rm Im}\ \eta>0, |\eta| e^{{\rm Re}\ \frac{1}{\eta^{2n} } }<\varepsilon\}.$$
Note also that we are interested in the local behaviour of $H_{n,\varepsilon}$ around $0$, so we will consider its intersection with $\{|\eta|<\delta\}$ for some small, fixed $\delta>0$. It is convenient to choose polar coordinates for $\eta$, $\eta=t e^{i\theta}$ for $t\in \mathbb R^+$ and $\theta\in [0,2\pi)$. With such a choice we have
$$\left | \frac{1}{\eta^{2n}}\right | = \frac{1}{t^{2n}},\ \ \arg{\left( \frac{1}{\eta^{2n}}\right )} = -2n\theta, $$
hence we can express $H_{n,\varepsilon}\cap \{|\eta|<\delta\}=D_{\varepsilon,\delta}$ as
$$D_{\varepsilon,\delta}=\{(t,\theta)\in \mathbb R^+ \times (0,\pi): t<\delta,\  t e^{\frac{\cos (2n\theta)}{t^{2n}}  } <\varepsilon \}.$$
By a direct inspection of the function $f_h(t) = t\exp(h/t^{2n})$, for a fixed $0<h\leq1$, one can see that $f_h(t)\to +\infty$ fot $t\to 0^+$ and $t \to +\infty$; moreover, it assumes minimum in $t(h) = (2nh)^{1/2n}$, and it is monotone on $\{0<t<t(h)\}$ and $\{t>t(h)\}$. Note that $t(h)\to 0$ for $h\to 0^+$. On the other hand, if $-1\leq h<0$ we have that $f_h(t)\to 0$ for $t\to 0^+$ and $f_h(t)\to +\infty$ for $t\to +\infty$; moreover, the function is monotone on $\mathbb R^+$.

We consider now the intersection of $D_{\varepsilon,\delta}$ with the half-lines $\{\theta=const\}$. If $\alpha_k = \pi/4n + (k/n)\pi$, $k=0,\ldots, n-1$, we have $\cos(2n\alpha_k) = 0$ and thus $D_{\varepsilon,\delta}\cap\{\theta=\alpha_k\} = \{t<\varepsilon\}$ (for $\varepsilon\leq \delta$). Exactly the same holds for $\{\theta=\beta_k\}$, where $\beta_k = 3\pi/4n + (k/n)\pi$, $k=0,\dots,n-1$. If $\alpha_k<\gamma<\beta_k$, we have that $\cos(2n\gamma)<0$, therefore $f_{\cos(2n\gamma)}(t)<t$ and $D_{\varepsilon,\delta}\cap\{\theta=\gamma\}$ is an interval containing $\{t<\varepsilon\}$.

When, instead, $\gamma\not\in (\alpha_k,\beta_k)$, computing $f_{\cos(2n\gamma)}(t)$ in $t=\varepsilon$ we see that $$f_{\cos(2n\gamma)}(\varepsilon) = \varepsilon e^{\frac{\cos(2n\gamma)}{\varepsilon^{2n}}} > \varepsilon$$
hence $D_{\varepsilon,\delta}\cap \{\theta = \gamma\}$ is an interval contained in $\{t<\varepsilon\}$ (whose extremes are continuous in $\gamma$).  Moreover, if we choose $\varepsilon < \sqrt[2n]{2n} \exp(1/2n) = f_1(t(1))$, it follows that $D_{\varepsilon,\delta}\cap \{\theta = \mu_k\} = \emptyset$ when $\mu_k = (k/n)\pi$, $k=0,\ldots,n-1$.  

From this last observation we derive that $D_{\varepsilon,\delta}\cap\{\theta > \mu_k\}$ is disconnected from $D_{\varepsilon,\delta}\cap \{\theta <\mu_k\}$ for $k=0,\ldots,n-1$, and the description provided in the previous paragraphs shows that $D_{\varepsilon,\delta}\cap \{\mu_k < \theta < \mu_{k+1}\}$ is connected. Hence $D_{\varepsilon,\delta}$ is made up by $n$ connected components $D_0,\ldots,D_{n-1}$, such that $D_k$ is contained in the sector $\{\mu_k<\theta<\mu_{k+1}\}$.
\end{proof} 
Since $\Phi$ is smooth up to the boundary and (since we are assuming $\Phi(0)=0$) sends $C$ into $C$, we have that for each $\varepsilon>0$ there exists $\delta$ such that $\phi(H_{n_1,\varepsilon})\subset H_{n_2,\delta}$ (and for each $\delta'>0$ there exists $\varepsilon'$ such that $\phi^{-1}(H_{n_2,\delta'})\subset H_{n_1,\varepsilon'}$). If $n_1<n_2$, Lemma \ref{conn} implies that for some connected component $H'$ of  $H_{n_2,\delta}$ we have $\phi^{-1}(H')\cap H_{n_1,\varepsilon}=\emptyset$. In particular, for all $\varepsilon'<\varepsilon$ we have $\phi^{-1}(H_{n_2,\delta})\not\subset H_{n_1,\varepsilon'}$, which is a contradiction.
\end{proof}

\end{document}